\documentclass[11pt]{amsart}
\usepackage[english]{babel}
\usepackage{amssymb}
\usepackage{amsmath}
\usepackage{amsthm}
\usepackage{bbold}
\usepackage{bbm}
\usepackage{mathtools}
\usepackage{dsfont}
\usepackage[scr=rsfso]{mathalfa}
\newtheorem{thm}{Theorem}
\newtheorem{prop}{Proposition}

\newtheorem{lem}{Lemma}
\newtheorem{exam}{Example}
\newtheorem{cor}{Corollary}

\newtheorem{definition}{Definition}

\numberwithin{equation}{section}
\newcommand{\abs}[1]{\lvert#1\rvert}
\newcommand{\norm}[1]{\lVert#1\rVert}

\newcommand{\oc}{\xrightarrow{\mathit{o}}}	
\newcommand{\ru}{\xrightarrow{\mathit{ru}}}	
\newcommand{\unc}{\xrightarrow{\mathit{un}}} 
\newcommand{\tc}{\xrightarrow{\tau}}	
\newcommand{\dc}{\xrightarrow{\mathit{d}}} 
\newcommand{\utc}{\xrightarrow{\mathit{u}\tau}} 
\newcommand{\thc}{\xrightarrow{\mathit{\hat{\tau}}}} 
\newcommand{\uoc}{\xrightarrow{\mathit{uo}}}	
\newcommand{\nc}{\xrightarrow{\norm{\cdot}}}	
\newcommand{\wc}{\xrightarrow{\mathit{w}}}	
\DeclareMathAlphabet{\mathpzc}{OT1}{pzc}{m}{it}
\DeclareSymbolFont{bbold}{U}{bbold}{m}{n}
\DeclareSymbolFontAlphabet{\mathbbold}{bbold}
\def\one{\mathbbold{1}}

\renewcommand{\le}{\leqslant}
\renewcommand{\ge}{\geqslant}

\begin{document}
	
\title{AMS Journal Sample}
	
	\author{Y. A. Dabboorasad$^{1,2}$}
	\address{$1$ Department of Mathematics, Islamic University of Gaza, P.O.Box 108, Gaza City, Palestine.}
	\email{yasad@iugaza.edu.ps, ysf\_atef@hotmail.com}

	\author{E. Y. Emelyanov$^2$}
	\author{M. A. A. Marabeh$^2$}
	\address{$^2$ Department of Mathematics, Middle East Technical University,  06800 Ankara, Turkey.}
	\email{yousef.dabboorasad@metu.edu.tr, eduard@metu.edu.tr, mohammad.marabeh@metu.edu.tr, m.maraabeh@gmail.com}
	
	
	
	\date{\today}
	
	\title{$u\tau$-Convergence in locally solid vector lattices}
	
\begin{abstract}
Let $x_\alpha$ be a net in a locally solid vector lattice $(X,\tau)$; we say that $x_\alpha$ is unbounded $\tau$-convergent to a vector $x\in X$ 
if $\lvert x_\alpha-x\rvert\wedge w\tc 0$ for all $w\in X_+$. In this paper, we study general properties of unbounded $\tau$-convergence (shortly, $u\tau$-convergence).
$u\tau$-Convergence generalizes unbounded norm convergence and unbounded absolute weak convergence in normed lattices 
that have been investigated recently. Besides, we introduce $u\tau$-topology and study briefly metrizabililty and completeness of this topology. 	
\end{abstract}
\maketitle

\section{Introduction and preliminaries}

The subject of ``unbounded convergence" has attracted many researchers \cite{Wick77,Tr04,Gao:14,GTX,EM,DOT,Zab,KMT,AEEM1,LC,KLT,GLX,Tay17}. 
It is well-investigated in vector lattices and normed lattices \cite{Gao:14,GaoX:14,GTX, Zab}. In the present paper, we study unbounded convergence 
in locally solid vector lattices. Results in this article extend previous works \cite{DOT,GTX,KMT,Zab}.

For a net $x_\alpha$ in a vector lattice $X$, we write
$x_\alpha\oc x$, if $x_\alpha$ {\em converges} to $x$ {\em in order}. 
This means that there is a net $y_\beta$, possibly over a
different index set, such that $y_\beta\downarrow 0$ and, for every
$\beta$, there exists $\alpha_\beta$ satisfying
$\abs{x_\alpha-x}\le y_\beta$ whenever $\alpha\ge\alpha_\beta$. 
A net $x_\alpha$ is {\em unbounded order convergent} to a vector 
$x\in X$ if $\abs{x_\alpha-x}\wedge u\oc 0$ for every $u\in X_+$. 
We write $x_\alpha\uoc x$ and say that $x_\alpha$ {\em uo-converges} to $x$. 
Clearly, order convergence implies $uo$-convergence and they coincide for order bounded nets. 
For a measure space $(\Omega,\Sigma,\mu)$ and for a sequence $f_n$ in $L_p(\mu)$ ($0\leq p\leq\infty$), $f_n\uoc 0$ iff $f_n\to 0$ almost everywhere 
(cf. \cite[Rem. 3.4]{GTX}). It is well known that almost everywhere convergence is not topological in general \cite{Ord66}. 
Therefore, the $uo$-convergence might not be topological. Quite recently, it has been shown that order convergence is never 
topological in infinite dimensional vector lattices \cite{DEM1}.

For a net $x_\alpha$ in a normed lattice $(X,\lVert\cdot\rVert)$, we write $x_\alpha\nc x$ if $x_\alpha$ converges to $x$ in norm. 
We say that $x_\alpha$ {\em unbounded norm converges} to $x\in X$ (or $x_\alpha$ {\em un-converges} to $x$) if $\abs{x_\alpha-x}\wedge u\nc 0$
for every $u\in X_+$. We write $x_\alpha\unc x$. Clearly, norm convergence implies $un$-convergence. The $un$-convergence is topological, 
and the corresponding topology (which is known as {\em un-topology}) was investigated in \cite{KMT}. 
A net $x_\alpha$ is {\em unbounded absolute weak convergent} to $x\in X$ (or $x_\alpha$ {\em $uaw$-converges} to $x$)  
if $\abs{x_\alpha-x}\wedge u\wc 0$ for all $u\in X_+$, where ``$\mathit{w}$" refers the weak convergence. 
We write $x_\alpha\xrightarrow{\mathit{uaw}}x$. Absolute weak convergence implies $uaw$-convergence. 
The notions of $\mathit{uaw}$-convergence and $uaw$-topology were introduced in \cite{Zab}.

If $X$ is a vector lattice, and $\tau$ is a linear topology on $X$ that has a base at zero consisting of solid sets, 
then the pair $(X,\tau)$ is called a {\em locally solid vector lattice}. It should be noted that all topologies considered throughout this article are assumed to be Hausdorff. 
It follows from \cite[Thm. 2.28]{Aliprantis:03} that a linear topology $\tau$ on a vector lattice $X$ is locally solid iff it is generated by a family $\{\rho_j\}_{j\in J}$ of Riesz pseudonorms. 
Moreover, if a family of Riesz pseudonorms generates a locally solid topology $\tau$ on a vector lattice $X$, then $x_\alpha\tc x$ in $X$ iff $\rho_j(x_\alpha-x)\xrightarrow[\alpha]{}0$ 
in $\mathbb{R}$ for each $j\in J$. Since $X$ is Hausdorff, then the family $\{\rho_j\}_{j\in J}$ of Riesz pseudonorms is separating; i.e., if $\rho_j(x)=0$ for all $j\in J$, then $x=0$. 
In this article, unless otherwise, the pair $(X,\tau)$ refers to as a locally solid vector lattice. 

A subset $A$ in a topological vector space $(X,\tau)$ is called {\em topologically bounded} (or simply $\tau$-{\em bounded}) 
if, for every $\tau$-neighborhood $V$ of zero, there exists some $\lambda>0$ such that $A\subseteq\lambda V$. 
If $\rho$ is a Riesz pseudonorm on a vector lattice $X$ and $x\in X$, then $\frac{1}{n}\rho(x)\leq\rho(\frac{1}{n}x)$ for all $n\in\mathbb{N}$. 
Indeed, if $n\in\mathbb{N}$ then $\rho(x)=\rho(n\frac{1}{n}x)\leq n\rho(\frac{1}{n}x)$. The following standard fact is included for the sake of completeness.

\begin{prop}
Let $(X,\tau)$ be a locally solid vector lattice with a family of a Riesz pseudonorms $\{\rho_j\}_{j\in J}$ that generates the topology $\tau$. 
If a subset $A$ of $X$ is $\tau$-bounded then $\rho_j(A)$ is bounded in $\mathbb{R}$ for any $j\in J$.
\end{prop}

\begin{proof}
Let $A\subseteq X$ be $\tau$-bounded and $j\in J$. Put $V:=\{x\in X:\rho_j(x)<1\}$. Clearly, $V$ is a neighborhood of zero in $X$. 
Since $A$ is $\tau$-bounded, there is $\lambda>0$ satisfying $A\subseteq\lambda V$. Thus $\rho_j(\frac{1}{\lambda}a)\leq1$ for all $a\in A$. 
There exists $n\in\mathbb{N}$ with $n>\lambda $. Now, $\frac{1}{n}\rho_j(a)\leq\rho_j(\frac{1}{n}a)\leq\rho_j(\frac{1}{\lambda}a)\leq1$ for all $a\in A$. 
Hence, $\sup_{a\in A}\rho_j(a)\leq n<\infty$.
\end{proof}

Next, we discuss the converse of the proposition above. 

Let $\{\rho_j\}_{j\in J}$ be a family of Riesz pseudonorms for a locally solid vector lattice $(X,\tau)$. 
For $j\in J$, let $\tilde{\rho}_j:=\frac{\rho_j}{1+\rho_j}$. Then $\tilde{\rho}_j$ is a Riesz pseudonorm on $X$. 
Moreover, the family $(\tilde{\rho}_j)_{j\in J}$ generates the topology $\tau$ on $X$. 
Clearly, $\tilde{\rho}_j(A)\leq 1$ for any subset $A$ of $X$, but still we might have a subset that is not $\tau$-bounded.

Recall that a locally solid vector lattice $(X,\tau)$ is said to have the {\em Lebesgue property} if $x_\alpha\downarrow 0$ in $X$ 
implies $x_\alpha\tc 0$; or equivalently $x_\alpha\oc 0$ implies $x_\alpha\tc 0$; and $(X,\tau)$ is said to have the {\em $\sigma$-Lebesgue property} 
if $x_n\downarrow 0$ in $X$ implies $x_n\tc 0$. Finally, $(X,\tau)$ is said to have the {\em Levi property} 
if $0\leq x_\alpha\uparrow$ and the net $x_\alpha$ is $\tau$-bounded, then $x_\alpha$ has the supremum in $X$; 
and $(X,\tau)$ is said to have the {\em $\sigma$-Levi property} if $0\leq x_n\uparrow$ and $x_n$ is $\tau$-bounded, then $x_n$ has supremum in $X$, see \cite[Def. 3.16]{Aliprantis:03}.

Let $X$ be a  vector lattice, and take  $0\neq u\in X_+$. Then a net $x_\alpha$ in $X$ is said to be {\em u-uniformly convergent} to a vector $x\in X$ if, for each $\varepsilon>0$, 
there exists some $\alpha_{\varepsilon}$ such that $\abs{x_\alpha-x}\leq\varepsilon u$ holds for all $\alpha\ge\alpha_{\varepsilon}$; and $x_\alpha$ is said to be  
{\em u-uniformly Cauchy} if, for each $\varepsilon > 0 $, there exists some $\alpha_{\varepsilon}$ such that, for all $\alpha,\alpha^\prime\ge\alpha_{\varepsilon}$,  
we have $\abs{x_\alpha-x_{\alpha^\prime}}\leq\varepsilon u$. A vector lattice $X$ is said to be {\em u-uniformly complete} 
if every $u$-uniformly Cauchy sequence in $X$ is $u$-uniformly convergent; and $X$ is said to be {\em uniformly complete} 
if $X$ is $u$-uniformly complete for each $0\neq u\in X_+$.

Let $X$ be a vector lattice. An element $0\neq e\in X_+$ is called a {\em strong unit} if $I_e=X$ (equivalently, 
for every $x\ge 0$, there exists $n\in\mathbb N$ 
such that $x\le ne$), and $0\neq e\in X_+$ is called a {\em weak unit} if $B_e=X$
(equivalently, $x\wedge ne\uparrow x$ for every $x\in X_+$). Here $B_e$ denotes the band generated by $e$. If $(X,\tau)$ is a topological vector lattice, 
then $0\neq e\in X_+$ is called a {\em quasi-interior point}, if the principal ideal $I_e$ is $\tau$-dense in $X$
 \cite[Def. II.6.1]{Schaefer:74}.
It is known that
\begin{displaymath}
\text{strong unit}\Rightarrow
\text{quasi-interior point}\Rightarrow
\text{weak unit}.
\end{displaymath}
Recall that a Banach lattice $X$ is called an {\em $AM$-space} if $\norm{x\vee y}=\max\{\norm{x},\norm{y}\}$ for all $x,y\in X$ with $x\wedge y=0$. 

Let $(X,\tau)$ be a sequentially complete locally solid vector lattice. Then it follows from the proof of \cite[Cor. 2.59]{Aliprantis_Tourky:07} that it is uniformly complete. 
So, for each $0\neq u\in X_+$, let $I_u$ be the ideal generated by $u$ and $\norm{\cdot}_u$ be the norm on $I_u$ given by  
$$
  \norm{x}_u=\inf\{r>0:\abs{x}\leq ru\}~~~~(x \in X).
$$ 

Then, by \cite[Thm. 2.58]{Aliprantis_Tourky:07}, the pair $(I_u,\norm{.}_u)$ is a Banach lattice. Now Theorem 3.4 in \cite{Abramovich:02} 
implies that $(I_u,\norm{\cdot}_u)$ is an $AM$-space with a strong unit $u$, and then, by \cite[Thm. 3.6]{Abramovich:02}, 
it is lattice isometric (uniquely, up to a homeomorphism) to $C(K)$ for some compact Hausdorff space $K$ in such a way, 
that the strong unit $u$ is identified with the constant function $\one$ on $K$.

For unexplained terminologies and notions we refer to \cite{Aliprantis:03,Aliprantis:06}.

\section{Unbounded $\tau$-convergence}

Suppose $(X,\tau)$ is a locally solid vector lattice. Let $x_\alpha$ be a net in $X$. 
We say that $x_\alpha$ is unbounded $\tau$-convergent to $x\in X$ if, for any $w\in X_+$, we have $\abs{x_\alpha-x}\wedge w\tc 0$.   
In this case, we write $x_\alpha\utc x$ and say that $x_\alpha$  {\em $u\tau$-converges to} $x$. Obviously, if $x_\alpha\tc x$ then $x_\alpha\utc x$. The converse holds if the net $x_\alpha$ is order bounded.
Note also that $u\tau$-convergence respects linear and lattice operations. It is clear that $u \tau$-convergence is a generalization of $un$-convergence \cite{DOT,KMT} and, of $uaw$-convergence \cite{Zab}.

Let $\mathcal{N}_\tau$ be a neighborhood base at zero consisting of solid sets for $(X,\tau)$. For each $0\neq w\in X_+$ and $V\in\mathcal{N}_\tau$, let
$$
  U_{V,w}\coloneqq\{x\in X:\abs{x}\wedge w\in V\}.
$$
It can be easily shown that the collection 
$$
  \mathcal{N}_{u\tau}\coloneqq\{U_{V,w}:V\in\mathcal{N}_\tau, 0\neq w\in X_+\}
$$
forms a neighborhood base at zero for a locally solid topology; we call it $u\tau$-topology, where $u$ refers to as \textit{unbounded}. 
Moreover, $x_\alpha\xrightarrow{\mathit{{u\tau}}}0$ iff $x_\alpha\rightarrow 0$ with respect to $u\tau$-topology. 
Indeed, suppose $x_\alpha\xrightarrow{\mathit{{u\tau}}}0$. Given a neighborhood $U_{V,w}\in\mathcal{N}_{u\tau}$. 
Then there are $0\neq w\in X_+$ and $V\in\mathcal{N}_\tau$ such that
$$
  U_{V,w}=\{x\in X:\abs{x}\wedge w\in V\}.
$$
Now, $x_\alpha\xrightarrow{\mathit{{u\tau}}}0$ implies $\abs{x_\alpha}\wedge w\tc 0$. 
So, there is $\alpha_0$ such that, for all $\alpha\geq\alpha_0$, we have $\abs{x_\alpha}\wedge w\in V$. 
That is $x_\alpha\in U_{V,w}$ for all $\alpha\geq\alpha_0$. Thus, $x_\alpha\rightarrow 0$ in the $u\tau$-topology.

Conversely, assume $x_\alpha\to 0$ in the $u\tau$-topology. Given $0\neq w\in X_+$ and $V\in\mathcal{N_\tau}$. 
Then, $U_{V,w}$ is a zero neighborhood in the $u\tau$-topology. So, there is $\alpha'$ such that $x_\alpha\in U_{V,w}$ for all $\alpha\geq\alpha'$. 
That is, $\abs{x_\alpha}\wedge w\in V$ for all $\alpha\geq\alpha'$. Thus, $\abs{x_\alpha}\wedge w\tc 0$ or $x_\alpha\xrightarrow{\mathit{{u\tau}}}0$.
The locally solid $u\tau$-topology will be referred to as {\em unbounded $\tau$-topology}.

The neighborhood base at zero for the $u\tau$-topology on $X$ has an equivalent representation in terms of 
a family $(\rho_j)_{j\in J}$ of Riesz pseudonorms that generates the topology $\tau$. For $\varepsilon>0$,  $j\in J$, and $0\neq w\in X_+$, let
$V_{\varepsilon,w,j}\coloneqq \{x\in X:\rho_j(\abs{x}\wedge w)<\varepsilon\}$. 
Clearly, the collection $\{V_{\varepsilon,w,j}:\varepsilon>0, 0\neq w\in X_+,j\in J\}$ generates the $u\tau$-topology.

It is known that the topology of any linear topological space can be derived from a unique translation-invariant uniformity, 
i.e., any linear topological space is uniformisable (cf. \cite[Thm. 1.4]{Schaefer:99}). 
It follows from \cite[Thm. 8.1.20]{Engelking:89} that any linear topological space is completely regular. 
In particular, the unbounded $\tau$-convergence is completely regular.

Since $x_\alpha\tc 0$ implies $x_\alpha\utc 0$, then the $\tau$-topology in general is finer than $u\tau$-topology. 
The next result should be compared with \cite[Lm. 2.1]{KMT}.

\begin{lem}\label{dichotomy}
Let $(X,\tau)$ be a sequentially complete locally solid  vector lattice, where $\tau$ is generated by a family $(\rho_j)_{j\in J}$ of Riesz pseudonorms. 
Let $\varepsilon>0$, $j\in J$, and $0\ne w\in X_+$. Then either $V_{\varepsilon,w,j}$ is contained in $[-w,w]$, or it contains a non-trivial ideal.
\end{lem}

\begin{proof}
Suppose that  $V_{\varepsilon,w,j}$ is not contained in $[-w,w]$. Then there exists $x\in V_{\varepsilon,w,j}$ such that
$x\not\in[-w,w]$. Replacing $x$ with $\abs{x}$, we may assume $x>0$. Since $x\not\in[-w,w]$, then  $y=(x-w)^+>0$. 
Now, letting $z=x\vee w$, we have that the ideal $I_z$ generated by $z$, is lattice and norm isomorphic to $C(K)$ for some compact and Hausdorff space $K$, 
where $z$ corresponds to the constant function $\mathbbm{1}$. Also $x$, $y$, and $w$ in $I_z$ correspond to $x(t)$, $y(t)$, and $w(t)$ in $C(K)$ respectively.
	
Our aim is to show that for all $\alpha\geq  0$ and $t\in K$, we have
$$
  (\alpha y)(t)\wedge w(t)\leq x(t)\wedge w(t).
$$
For this, note that $y(t)=(x-w)^+(t)=(x-w)(t)\vee 0$.
	
Let $t\in K$ be arbitrary.
\begin{itemize}
		\item Case (1): If $(x-w)(t)>0$, then $x(t)\wedge w(t)=w(t)\geq(\alpha y)(t)\wedge w(t)$ for all $\alpha\geq 0 $, as desired.
		\item Case (2): If $(x-w)(t)<0$, then $(\alpha y)(t)\wedge w(t)\leq(\alpha y)(t)=\alpha(x-w)(t)\vee 0=0\leq x(t)\wedge w(t)$, as desired.
\end{itemize}
Hence, for all $\alpha\geq 0$ and $t\in K$, we have $(\alpha w)(t)\wedge w(t)\leq x(t)\wedge w(t)$ and so $(\alpha y)\wedge w\leq x\wedge w$ for all $\alpha\geq 0$ . 
Note, that  $\alpha y, w, x\in X_+$. Thus $\rho_j(\left|\alpha y\right|\wedge w)\leq\rho_j(\left|x\right|\wedge w)<\varepsilon$, so $\alpha y\in V_{\varepsilon,w,j}$ 
and, since $V_{\varepsilon,w,j}$ is solid, then $I_z\subseteq V_{\varepsilon,w,j}$.
\end{proof}

Note that the sequential completeness in Lemma \ref{dichotomy} can be removed, as we see in the following corollary.

\begin{thm}
	Let $(X,\tau)$ be a locally solid  vector lattice, where $\tau$ is generated by a family $(\rho_j)_{j\in J}$ of Riesz pseudonorms. Let $\varepsilon>0$, $j\in J$, and $0\ne w\in X_+$. 
	Then either $V_{\varepsilon,w,j}$ is contained in $[-w,w]$ or $V_{\varepsilon,w,j}$ contains a non-trivial ideal.
\end{thm}

\begin{proof}
Given $\varepsilon>0$, $j\in J$, and $0\ne w\in X_+$. Let $(\hat{X},\hat{\tau})$ be the topological completion of $(X,\tau)$. 
In particular, $(\hat{X},\hat{\tau})$ is  sequentially complete. Let ${\hat{V}}_{\varepsilon,w,j}=\{\hat{x}\in\hat{X}\colon{\hat{\rho}}_j(\abs{\hat{x}}\wedge w)<\varepsilon\}$. 
Then $V_{\varepsilon,w,j}=X\cap{\hat{V}}_{\varepsilon,w,j}$. By Lemma \ref{dichotomy}, either ${\hat{V}}_{\varepsilon,w,j}$ is a subset of $[-w,w]_{\hat{X}}$ in $\hat{X}$ or 
${\hat{V}}_{\varepsilon,w,j}$ contains a non-trivial ideal of $\hat{X}$. If ${\hat{V}}_{\varepsilon,w,j}\subseteq [-w,w]_{\hat{X}}$, then 
$$
  V_{\varepsilon,w,j}=X\cap{\hat{V}}_{\varepsilon,w,j}\subseteq X\cap[-w,w]_{\hat{X}}=[-w,w]\subseteq X. 
$$
If ${\hat{V}}_{\varepsilon,w,j}$ contains a non-trivial ideal, then ${\hat{V}}_{\varepsilon,w,j}\nsubseteq[-w,w]_{\hat{X}}$. 
So, there is $\hat{x}\in\hat{V}_{\varepsilon,w,j}$ with $\hat{x}\notin[-w,w]_{\hat{X}}$. 
Since $[-w,w]_{\hat{X}}$ is $\hat{\tau}$-closed, then there is a solid neighborhood ${N}_{\hat{x}}$ of $\hat{x}$ in $\hat{X}$ 
such that ${N}_{\hat{x}}\cap[-w,w]_{\hat{X}}=\emptyset$. Hence, ${N}_{\hat{x}}\cap{\hat{V}}_{\varepsilon,w,j}\cap[-w,w]_{\hat{X}}=\emptyset$, and  
$N_{\hat{x}}\cap{\hat{V}}_{\varepsilon,w,j}$ is open in $\hat{X}$ with $\hat{x}\in{N}_{\hat{x}}\cap{\hat{V}}_{\varepsilon,w,j}$. 
By $\tau$-density of $X$ in $\hat{X}$, we may take $x\in X\cap{N}_{\hat{x}}\cap{\hat{V}}_{\varepsilon,w,j}$. 
Since $\abs{x}\in X\cap{N}_{\hat{x}}\cap{\hat{V}}_{\varepsilon,w,j}$, we may also assume that $x\in X_+$. 
	
Let $y:=(x-w)^+$, then $y>0$ and $y\in X_+$. By the same argument in Lemma \ref{dichotomy}, we get $(\alpha y)\wedge w\leq x\wedge w$ for all $\alpha\in\mathbb{R}_+$. 
Since $x\in{\hat{V}}_{\varepsilon,w,j}$, then $\alpha y\in{\hat{V}}_{\varepsilon,w,j}$ for all $\alpha\in\mathbb{R}_+$. 
	But $\alpha y\in X_+$ for all $\alpha\in\mathbb{R}_+$ and, since $V_{\varepsilon,w,j}=X\cap{\hat{V}}_{\varepsilon,w,j}$, 
	we get $\alpha y\in V_{\varepsilon,w,j}$ for all $\alpha\in\mathbb{R}_+$. Since $V_{\varepsilon,w,j}$ is solid, 
	we conclude that the principal ideal $I_y$ taken in $X$ is a subset of $V_{\varepsilon,w,j}$.
\end{proof}

\begin{lem}\label{V-bdd-su}
Let $(X,\tau)$ be a locally solid  vector lattice, where $\tau$ is generated by a family $(\rho_j)_{j\in J}$ of Riesz pseudonorms. 
If $V_{\varepsilon,w,j}$ is contained in $[-w,w]$, then $w$ is a strong unit.
\end{lem}

\begin{proof}
Suppose $V_{\varepsilon,w,j}\subseteq[-w,w] $. Since  $V_{\varepsilon,w,j}$ is absorbing,  for any $x\in X_+$, 
there exist $\alpha>0$ such that $\alpha x\in V_{\varepsilon,w,j}$ , and so $\alpha x\in[-w,w]$, or $x\leq\frac{1}{\alpha}w$. 
Thus $w$ is a strong unit, as desired. 
\end{proof}

\begin{prop}\label{q.i.p in X and its completion}
	Let $e\in X_+$.  Then $e$ is a quasi-interior point in $(X,\tau)$ iff $e$ is a quasi-interior point in the topological completion $(\hat{X},\hat{\tau})$.
\end{prop}

\begin{proof}\
	The backward implication is trivial.\\
	For the forward implication let $\hat{x}\in\hat{X}_+$. Our aim is to show that $\hat{x}-\hat{x}\wedge ne\tc 0$ in $\hat{X}$ as $n\to\infty$. By \cite[Thm. 2.40]{Aliprantis:03}, 
	$\hat{X}_+={\overline{X}^{\hat{\tau}}_+}$. So, there is a net $x_\alpha$ in $X_+$ such that $x_\alpha\thc\hat{x}$ in $\hat{X}$.
	Let $j\in J$ and $\varepsilon>0$. Since $\hat{\rho}_j(x_\alpha-\hat{x})\to 0$, then there is $\alpha_\varepsilon$ satisfying
\begin{equation}\label{2}
	\hat{\rho}_j(x_{\alpha_\varepsilon}-\hat{x})<\varepsilon.
\end{equation}
Since $e$ is a quasi-interior point in $X$ and $x_{\alpha_\varepsilon}\in X_+$, then $x_{\alpha_\varepsilon}-x_{\alpha_\varepsilon}\wedge ne\tc 0$ in $X$ as $n\to\infty$. 
Thus, there is $n_\varepsilon\in\mathbb{N}$ such that
\begin{equation}\label{3}
	\hat{\rho}_j(x_{\alpha_\varepsilon}-ne\wedge x_{\alpha_\varepsilon})=\rho_j(x_{\alpha_\varepsilon}-ne\wedge x_{\alpha_\varepsilon})<\varepsilon ~~~~ (\forall n\ge n_\varepsilon).
\end{equation}
Now, $0\leq\hat{x}-\hat{x}\wedge ne=\hat{x}-x_{\alpha_\varepsilon}+x_{\alpha_\varepsilon}-ne\wedge x_{\alpha_\varepsilon}+ne\wedge x_{\alpha_\varepsilon}-\hat{x}\wedge ne$. 
So $\hat{\rho}_j(\hat{x}-\hat{x}\wedge ne)\leq\hat{\rho}_j(\hat{x}-x_{\alpha_\varepsilon})+\hat{\rho}_j(x_{\alpha_\varepsilon}-ne\wedge x_{\alpha_\varepsilon})+\hat{\rho}_j(ne\wedge x_{\alpha_\varepsilon}-\hat{x}\wedge ne)$. 
For $n\ge n_\varepsilon$, we have, by \eqref{2}, \eqref{3}, and \cite[Thm. 1.9(2)]{Aliprantis:06}, that
$$
  \hat{\rho}_j(\hat{x}-\hat{x}\wedge ne)\leq\varepsilon+\varepsilon+\hat{\rho}_j(x_{\alpha_\varepsilon}-\hat{x})\leq 3\varepsilon. 
$$ 
Therefore, $e$ is a quasi-interior point in $\hat{X}$.
\end{proof}	

The technique used in the proof of \cite[Thm. 3.1]{KMT} can be used in the following theorem as well, and so we omit its proof.

\begin{thm}\label{qip is qip in completion}
Let $(X,\tau)$ be a sequentially complete locally solid  vector lattice, where $\tau$ is generated by a family $(\rho_j)_{j\in J}$ of Riesz pseudonorms. Let $e\in X_+$. The following are equivalent:
	\begin{enumerate}
		\item\label{qip for ut convergence} $e$ is a quasi-interior point;
		\item\label{qip for ut convergence net version} for every net $x_\alpha$ in $X_+$, if $x_\alpha\wedge e\tc 0$ then $x_\alpha\utc 0$;
		\item\label{qip-seq} for every sequence $x_n$ in $X_+$, if $x_n\wedge e\tc 0$ then $x_n\utc 0$. 
	\end{enumerate}
\end{thm}

\section{Unbounded $\tau$-convergence in sublattices}	

Let $Y$ be a sublattice of a locally solid vector lattice $(X,\tau)$. If $y_\alpha$ is a net in $Y$ such that $y_\alpha \utc 0$ in $X$, then clearly, $y_\alpha \utc 0$ in $Y$. The converse does not hold in general. For example, the sequence $e_n$ of standard unit vectors is $un$-null in $c_0$, but not in $\ell_\infty$. In this section, we study when the $u \tau$-convergence passes from a sublattice to the whole space.

Recall that a sublattice $Y$ of a vector lattice $X$ is {\em majorizing} if, for every $x\in X_+$, there exists $y\in Y_+$ with $x\le y$. 
The following theorem extends \cite[Thm. 4.3]{KMT} to locally solid vector lattices.

\begin{thm}\label{Um convergence in sublattices}
Let $(X,\tau)$ be a locally solid vector lattice and $Y$ be a sublattice of $X$. If $y_\alpha$ is a net in $Y$ and $y_\alpha\utc 0$ in $Y$, then $y_\alpha\utc 0$ in $X$ in each of the following cases$:$
\begin{enumerate}
		\item\label{majorizing sblattice} $Y$ is majorizing in $X$$;$
		\item $Y$\label{tau-dense sublattice} is $\tau$-dense in $X$$;$
		\item\label{projection band sublattice} $Y$ is a projection band in $X$.
\end{enumerate}
\end{thm}

\begin{proof}
\begin{enumerate}
  	\item Trivial.
		\item Let $u\in X_+$. Fix $\varepsilon>0$ and take $j\in J$. Since $Y$ is $\tau$-dense in $X$, then there is $v\in Y_+$ such that 
		$\rho_j(u-v)<\varepsilon$. But $y_\alpha\utc 0$ in $Y$ and so, in particular, $\rho_j(\abs{y_\alpha}\wedge v)\to 0$. So there is $\alpha_0$ 
		such that $\rho_j(\abs{y_\alpha}\wedge v)<\varepsilon$ for all $\alpha\ge\alpha_0$. 
		It follows from $u\leq v+\abs{u-v}$, that $\abs{y_\alpha}\wedge u\leq\abs{y_\alpha}\wedge v+\abs{u-v}$, and so $\rho_j(\abs{y_\alpha}\wedge u)<\rho_j(\abs{y_\alpha}\wedge v)+\rho_j(u-v)<2\varepsilon$. 
		Thus, $\rho_j(\abs{y_\alpha}\wedge u)\to 0$ in $\mathbb{R}$. Since $j\in J$ was chosen arbitrary, we conclude that $y_\alpha\utc 0$ in $X$.
		\item Let $u\in X_+$. Then $u=v +w$, where $v\in Y_+$ and $w\in Y_+^d$. 
		Now $\abs{y_\alpha}\wedge u=\abs{y_\alpha}\wedge v+\abs{y_\alpha}\wedge w=\abs{y_\alpha}\wedge v$, since $y_\alpha\in Y$. 
		Then $\abs{y_\alpha}\wedge u=\abs{y_\alpha}\wedge v\tc 0$ in $X$.
\end{enumerate}
\end{proof}
	
\begin{cor}
If $(X,\tau)$ is a locally solid vector lattice and $x_\alpha\utc 0$ in $X$, then $x_\alpha\utc 0$ in the Dedekind completion $X^\delta$ of $X$.
\end{cor}
	
\begin{cor}
If $(X,\tau)$ is a locally solid vector lattice and $x_\alpha\utc 0$ in $X$, then $x_\alpha\utc 0$ in the topological completion $\hat{X}$ of $X$.
\end{cor}

The next result generalizes Corollary 4.6 in \cite{KMT} and Proposition 16 in \cite{Zab}. 

\begin{thm}\label{u-convergence from a sublattice}
Let $(X,\tau)$ be a topologically complete locally solid vector lattice that possesses the Lebesgue property, and $Y$ be a sublattice of $X$. 
If $y_\alpha\utc 0$ in $Y$, then $y_\alpha\utc 0$ in $X$.
\end{thm}

\begin{proof}
Suppose  $y_\alpha\utc 0$ in $Y$. By Theorem \ref{Um convergence in sublattices}(\ref{majorizing sblattice}), $y_\alpha\utc 0$ in the ideal $I(Y)$ generated by $Y$ in $X$. 
By Theorem \ref{Um convergence in sublattices}(\ref{tau-dense sublattice}), $y_\alpha\utc 0$ in the closure $\overline{\{I(Y)\}}^\tau$ of $I( Y)$. 
It follows from \cite[Thm. 3.7]{Aliprantis:03} that $\overline{\{I(Y)\}}^\tau$ is a band in $X$. Now, \cite[Thm. 3.24]{Aliprantis:03} 
assures that $X$ is Dedekind complete, and so $\overline{\{I(Y)\}}^\tau$ is a projection band in $X$. 
Then $y_\alpha\utc 0$ in $X$, in view of Theorem \ref{Um convergence in sublattices}(\ref{projection band sublattice}).
\end{proof}

Suppose that $(X,\tau)$ is a locally solid vector lattice possessing the Lebesgue property. Then, in view of \cite[Thms. 3.23 and 3.26]{Aliprantis:03}, 
its topological completion $(\hat{X},\hat{\tau})$ possesses the Lebesgue property as well. Hence, by \cite[Thm. 3.24]{Aliprantis:03}, $\hat{X}$ is Dedekind complete. Since $X\subseteq \hat{X}$, 
there holds $X^\delta\subseteq(\hat{X})^\delta=\hat{X}$. So, $X\subseteq X^\delta\subseteq\hat{X}$. Now, Theorem \ref{u-convergence from a sublattice} assures that, given a net 
$z_\alpha$  in $X^\delta$, if  $z_\alpha\utc 0$ in $X^\delta$ then $z_\alpha\utc 0$ in $\hat{X}$.\\

\section{unbounded relatively uniformly convergence}

In this section we discuss unbounded relatively uniformly convergence. Recall that a net $x_\alpha$ in a vector lattice $X$ is said to be {\em relatively uniformly convergent} to $x \in X$ if, there is $u \in X_+$ such that for any $n \in \mathbb{N}$, there exists $\alpha_n$ satisfying $\abs{x_\alpha - x} \leq \frac{1}{n}u$ for $\alpha \ge \alpha_n$. In this case we write $x_\alpha \ru x$ and the vector $u \in X_+$ is called {\em regulator}, see \cite[Def. III.11.1]{Vulikh67}.

If $x_\alpha\xrightarrow{\mathit{{ru}}}0$ in a locally solid vector lattice $(X,\tau)$, 
then $x_\alpha\xrightarrow{\mathit{{\tau}}}0$. Indeed, let $V$ be a solid neighborhood at zero. Since $x_\alpha\xrightarrow{\mathit{{ru}}}0$, then there is $u\in X_+$ such that, for a given $\varepsilon>0$, there is $\alpha_\varepsilon$ satisfying $\abs{x_\alpha}\leq\varepsilon u$ for all $\alpha\geq\alpha_\varepsilon$.	Since $V$ is absorbing, there is $c\geq 1$ such that $\frac{1}{c}u\in V$. There is some $\alpha_0$ such that $\abs{x_\alpha}\leq\frac{1}{c}u$ for all $\alpha\geq\alpha_0$. Since $V$ is solid and $\abs{x_\alpha}\leq\frac{1}{c}u$ for all $\alpha\geq\alpha_0$, then $x_\alpha\in V$ for all $\alpha\geq\alpha_0$. That is $x_\alpha\xrightarrow{\mathit{{\tau}}}0$.

The following result might be considered as an $ru$-version of Theorem 1 in \cite{DEM1}.

\begin{thm}\label{when ru-convergence is topological}
Let $X$ be a vector lattice. Then the following conditions are equivalent.

$(1)$ There exists a linear topology $\tau$ on $X$ such that, for any net $x_\alpha$ in $X$: $x_\alpha\ru 0$ iff $x_\alpha\xrightarrow{\mathit{{\tau}}}0$. 

$(2)$ There exists a norm $\|\cdot\|$ on $X$ such that, for any net $x_\alpha$ in $X$: $x_\alpha\ru 0$ iff $\|x_\alpha\|\to 0$. 

$(3)$ $X$ has a strong order unit.
\end{thm}

\begin{proof}
$(1)\Rightarrow(3)$ It follows from \cite[Lem. 1]{DEM1}.

$(3)\Rightarrow(2)$ Let $e\in X$ be a strong order unit. Then $x_\alpha\ru 0$ iff $\|x_\alpha\|_e\to 0$, where $\|x\|_e:=\inf\{r: |x|\le re\}$.

$(2)\Rightarrow(1)$ It is trivial.
\end{proof}

Let $X$ be a vector lattice. A net $x_\alpha$ in $X$ is said to be {\em unbounded relatively uniformly convergent} to $x\in X$ 
if $\abs{x_\alpha-x}\wedge w\xrightarrow{\mathit{{ru}}}0$ for all $w\in X_+$. In this case, we write $x_\alpha\xrightarrow{\mathit{{uru}}}x$.
Clearly, if $x_\alpha\xrightarrow{\mathit{{uru}}}0$ in a locally solid vector lattice $(X,\tau)$, 
then $x_\alpha\xrightarrow{\mathit{{u\tau}}}0$.

In general, $uru$-convergence is also not topological. Indeed, consider the vector lattice $L_1[0,1]$. 
It satisfies the diagonal property for order convergence by \cite[Thm. 71.8]{Riesz Space 1}. 
Now, by combining Theorems 16.3, 16.9, and 68.8 in \cite{Riesz Space 1} we get that for any sequence $f_n$ in $L_1[0,1]$
$f_n\oc 0$ iff $f_n\xrightarrow{\mathit{{ru}}}0$. In particular, $f_n\xrightarrow{\mathit{{uo}}}0$ iff $f_n\xrightarrow{\mathit{{uru}}}0$. 
But the $uo$-convergence in $L_1[0,1]$ is equivalent to $a.e$.-convergence which is not topological, see \cite{Ord66}.

However, in some vector lattices the $uru$-convergence could be topological. For example, if $X$ is a vector lattice with a strong unit $e$, 	
It follows from Theorem \ref{when ru-convergence is topological}, that $ru$-convergence is equivalent to the norm convergence $\lVert\cdot\lVert_e$, 
where $\lVert x\lVert_e\coloneqq\inf\{\lambda>0:\abs{x}\leq\lambda e\}$, $x\in X$. Thus $uru$-convergence in $X$ is topological.

Consider vector lattice $c_{00}$ of eventually zero sequences. 
It is well known that in $c_{00}$:  $x_\alpha\xrightarrow{\mathit{{ru}}}0$ iff $x_\alpha\oc 0$. 
For the sake of completeness we include a proof of this fact. Clearly, $x_\alpha\xrightarrow{\mathit{{ru}}}0\Rightarrow x_\alpha\oc 0$.
For the converse, suppose $x_\alpha\oc 0$ in $c_{00}$. Then there is a net $y_\beta\downarrow 0$ in $c_{00}$ 
such that, for any $\beta$, there is  $\alpha_\beta$ satisfying $\abs{x_\alpha}\leq y_\beta$ for all $\alpha\geq\alpha_\beta$. 
Let $e_n$ denote the sequence of standard unit vectors in $c_{00}$.
Fix $\beta_0$. Then $y_{\beta_0}=c_1^{\beta_0}e_{k_1}+\dots+c_n^{\beta_0}e_{k_n}, \ c_i^{\beta_0}\in\mathbb{R}, i=1,\dots,n$. 
Since $y_\beta$ is decreasing, then $y_\beta\leq y_{\beta_0}$ for all $\beta\geq\beta_0$. So, $y_\beta=c_1^{\beta}e_{k_1}+\dots+c_n^{\beta}e_{k_n}$ 
for all $\beta\geq\beta_0, c_i^{\beta}\in\mathbb{R}, i=1,\dots,n$.
Since $y_\beta\downarrow 0$ then $\lim_{\beta}c_i^{\beta}=0$ for all $i=1,\dots,n$. 
Let $u=e_{k_1}+\dots+e_{k_n}$. Given $\varepsilon>0$. 
Then, there is $\beta_\varepsilon\geq\beta_0$ such that $c_i^\beta<\varepsilon$ for all $\beta\geq\beta_\varepsilon$ for $i=1,\dots,n$.
Consider $y_{\beta_\varepsilon}$ then there is $\alpha_\varepsilon$ such that $\abs{x_\alpha}\leq y_{\beta_\varepsilon}$ for all $\alpha\geq\beta_\varepsilon$. 
But $y_{\beta_\varepsilon}=c_1^{\beta_\varepsilon}e_{k_1}+\dots+c_n^{\beta_\varepsilon}e_{k_n}\leq\varepsilon u$. 
So, $\abs{x_\alpha}\leq\varepsilon u$ for all $\alpha\geq\alpha_{\varepsilon}$. That is $x_\alpha\xrightarrow{\mathit{{ru}}}0$.
Thus, the $uru$-convergence in $c_{00}$ coincides with the $uo$-convergence which is pointwise convergence and, therefore, is topological.

\begin{prop}\label{In Lebesgue and complete metrizable ru is equivalent to tau}
	Let $X$ be Lebesgue and complete metrizable locally solid vector lattice. then
	$x_\alpha\xrightarrow{\mathit{{ru}}}0$ iff $x_\alpha\oc 0$.
\end{prop}

\begin{proof}
	The necessity is obvious. For the sufficiency assume that $x_\alpha\oc 0$.
	Then there exists $y_\beta\downarrow 0$ such that for any $\beta$ there is 
	$\alpha_{\beta}$ with $|x_\alpha|\le y_\beta$ as $\alpha\ge\alpha_{\beta}$.
	Since $d(y_\beta,0)\to 0$, there exists an increasing sequence $(\beta_k)_k$
	of indeces with $d(ky_{\beta_k},0)\le\frac{1}{2^k}$.
	Let $s_n=\sum_{k=1}^n ky_{\beta_k}$. We show the sequence $s_n$ is Cauchy. For $n > m$,
	\begin{align*}
	d(s_n,s_m) = d(s_n-s_m,0)=d\Big(\sum_{k=m+1}^n ky_{\beta_k},0\Big) &\leq \sum_{k=m+1}^n d\big(ky_{\beta_k},0\big)\\ &\leq \sum_{k=m+1}^n \frac{1}{2^k} \to 0, \text{ as } n,m \to \infty.
	\end{align*}
	Since $X$ is complete, then the sequence $s_n$ converges to some $u \in X_+$. That is, $u:=\sum\limits_{k=1}^{\infty}ky_{\beta_k}$. Then
	$$
	  k|x_\alpha|\le ky_{\beta_k}\le u \ \ \ \ (\forall \alpha\ge\alpha_{\beta_k})
	$$ 
	which means that $x_\alpha\xrightarrow{\mathit{{ru}}}0$.
\end{proof}

Let $X=\mathbb{R}^\Omega$ be the vector lattice of all real-valued functions on a set $\Omega$.

\begin{prop}\label{ru-convergence for sequences}
In the vector lattice $X=\mathbb{R}^\Omega$, the following conditions are equivalent$:$

$(1)$ for any net $f_\alpha$ in $X$: $f_\alpha\oc 0$ iff $f_\alpha\xrightarrow{\mathit{{ru}}}0$$;$ 

$(2)$ $\Omega$ is countable.
\end{prop}

\begin{proof}
$(1)\Rightarrow(2)$
Suppose $f_\alpha \oc 0\Leftrightarrow f_\alpha \xrightarrow{\mathit{{ru}}}0$ for any sequence $f_\alpha$ in $X=\mathbb{R}^\Omega$. 
Our aim is to show that $\Omega$ is countable. Assume, in contrary, that $\Omega$ is uncountable. Let $\mathcal{F}(\Omega)$ be the collection of all finite subsets of $\Omega$. 
For each $\alpha\in\mathcal{F}(\Omega)$, put $f_\alpha=\mathcal{X}_\alpha$. Clearly, $f_\alpha\uparrow\one$, where $\one$ denotes the constant function one on $\Omega$. 
Then $\one-f_\alpha\downarrow 0$ or $\one-f_\alpha\oc 0$ in $\mathbb{R}^\Omega$. 
So, there is $0\leq g\in\mathbb{R}^\Omega$ such that, for any $\varepsilon>0$, there exists $\alpha_\varepsilon$ 
satisfying $\one-f_\alpha\leq\varepsilon g$ for all $\alpha\ge\alpha_{\varepsilon}$. 
Let $n\in\mathbb{N}$. Then there is a finite set $\alpha_n\subseteq\Omega$ such that $\one-f_{\alpha_n}\leq\frac{1}{n}g$. 
Consequently,  $g(x)\ge n$ for all $x\in\Omega\setminus\alpha_n$. 
Let $S=\cup_{n=1}^\infty\alpha_n$. Then $S$ is countable and $\Omega\setminus S\neq \emptyset $. 
Moreover, for each $x\in\Omega\setminus S$, we have $g(x)\ge n$ for all $n\in\mathbb{N}$, which is impossible. 

$(2)\Rightarrow(1)$
Suppose that $\Omega$ is countable. So, we may assume that $X=s$, the space of all sequences. Since, from $x_\alpha\xrightarrow{\mathit{{ru}}}0$ always follows that $x_\alpha\oc 0$,
it is enough to show that if $x_\alpha\oc 0$ then $x_\alpha\xrightarrow{\mathit{{ru}}}0$. To see this, let $(x_\alpha^n)_n=x_\alpha\oc 0$.
Then, the net $x_\alpha$ is eventually bounded, say $|x_\alpha|\le u=(u_n)_n\in s$. Take $w:=(nu_n)_n\in s$. We show that 
$x_\alpha\xrightarrow{\mathit{{ru}}}0$ with the regulator $w$.
Let $k \in \mathbb{N}$. Since $x_\alpha \oc 0$, then for each $n \in \mathbb{N}$, $x_\alpha^n \to 0$ in $\mathbb{R}$. Hence, there is $\alpha_k$ such that $k \abs{x_\alpha^1} < u_1$, $k \abs{x_\alpha^2} < u_2$, $\cdots$, $k \abs{x_\alpha^{k-1}} < u_{k-1}$ for all $\alpha \ge \alpha_k$.
Note that for $n \ge k$, $k \abs{x_\alpha^n} < u_n$. Therefore, $k \abs{x_\alpha} < w$ for all $\alpha \ge \alpha_k$.
\end{proof}

It follows from Proposition \ref{ru-convergence for sequences} that, for countable $\Omega$, the $uru$-convergence in $\mathbb{R}^\Omega$ 
coincides with the $uo$-convergence (which is pointwise) and therefore is topological. 
We do not know, whether or not the countability of $\Omega$ is necessary
for the property that $uru$-convergence is topological in $\mathbb{R}^\Omega$.

\section{Topological orthogonal systems and metrizabililty}

A collection $\{e_\gamma\}_{\gamma\in\Gamma}$ of positive vectors in a  vector lattice $X$ is called an {\em orthogonal system} if $e_\gamma\wedge e_{\gamma'}=0$ 
for all $\gamma\neq\gamma'$. If, moreover, $x\wedge e_\gamma=0$ for all $\gamma\in\Gamma$ implies $x=0$, then $\{e_\gamma\}_{\gamma\in\Gamma}$ 
is called a {\em maximal orthogonal system}.
It follows from Zorn's Lemma that every  vector lattice containing at least one non-zero element has a maximal orthogonal system. 
Motivated by Definition III.5.1 in \cite{Schaefer:74}, we introduce the following notion.

\begin{definition}
	Let $(X,\tau)$ be a topological vector lattice. An orthogonal system $Q=\{e_\gamma\}_{\gamma\in\Gamma}$ of non-zero elements in $X_+$ is said to be a 
	{\em topological orthogonal system} if the ideal $I_Q$ generated by $Q$ is $\tau$-dense in $X$.
\end{definition}

\begin{lem}\label{tos implies mos}
	If $Q=\{e_\gamma\}_{\gamma\in\Gamma}$ is a topological orthogonal system in a topological vector lattice $(X,\tau)$, then $Q$ is a maximal orthogonal system in $X$.
\end{lem}

\begin{proof}
	Assume $x\wedge e_\gamma=0$ for all $\gamma\in\Gamma$. By the assumption, there is a net $x_\alpha$ in the ideal $I_Q$ such that $x_\alpha\tc x$. 
	Without lost of generality, we may assume $0\leq x_\alpha\leq x$ for all $\alpha$. Since  
	$x_\alpha\in I_Q$, then there are $0<\mu_\alpha\in\mathbb{R}$ and $\gamma_1,~\gamma_2,\dots,~\gamma_n$, 
	such that $0\leq  x_\alpha\leq\mu_\alpha(e_{\gamma_1}+e_{\gamma_2} +\dots + e_{\gamma_n})$. 
	So $0\leq  x_\alpha=x_\alpha\wedge x\leq\mu_\alpha(e_{\gamma_1}+e_{\gamma_2}+\dots+e_{\gamma_n})\wedge x$ 
	$=\mu_\alpha e_{\gamma_1}\wedge x+\dots+\mu_\alpha e_{\gamma_n}\wedge x$ $=0$. 
	Hence $x_\alpha =0$ for all $\alpha$, and so $x=0$.
\end{proof}

We recall the following construction from \cite[p.169]{Schaefer:74}. 
Let $X$ be a  vector lattice and $Q=\{e_\gamma\}_{\gamma\in\Gamma}$ be a maximal orthogonal system of $X$. 
Let ${\mathscr F}(\Gamma)$ denote the collection of all finite subsets of $\Gamma$ ordered by inclusion. 
For each $(n,H)\in\mathbb{N}\times\mathscr F(\Gamma)$ and $x\in X_+$, define
$$
x_{n,H}\coloneqq\sum_{\gamma\in H}x\wedge n e_\gamma.
$$
Clearly $\{x_{n,H}:(n,H)\in\mathbb{N}\times\mathscr F(\Gamma)\}$ is directed upward, and
\begin{equation}\label{star}
x_{n,H}\leq x\quad\text{for all}\quad (n,H)\in\mathbb{N}\times\mathscr F(\Gamma).
\end{equation}
Moreover, Proposition II.1.9 in \cite{Schaefer:74} implies $x_{n,H}\uparrow x$.

\begin{thm}\label{Proposition II.1.9 in Sch}
	Let $Q=\{e_\gamma\}_{\gamma\in\Gamma}$ be an orthogonal system of a locally solid vector lattice $(X,\tau)$, 
	then $Q$ is a topological orthogonal system iff we have $x_{n,H}\tc x$ over $(n,H)\in\mathbb{N}\times\mathscr F(\Gamma)$ for each $x\in X_+$.
\end{thm}

\begin{proof}\
	For the backward implication take $x\in X_+$. Since 
	$$
	x_{n,H}=\sum\limits_{\gamma\in H}x\wedge n e_\gamma\leq n\sum\limits_{\gamma\in H}e_\gamma,
	$$ 
	then $x_{n,H}\in I_Q$ for each $(n,H)\in\mathbb{N}\times\mathscr F(\Gamma)$. Also, we have, by assumption, $x_{n,H}\tc x$. 
	Thus, $x\in\overline{I}^\tau_Q$, i.e., $Q$ is a topological orthogonal system of $X$.
	
	For the forward implication, note that $Q$ is a maximal orthogonal system, by Lemma \ref{tos implies mos}. Let $x\in X_+$, and $j\in J$. 
	Given $\varepsilon>0$. Let $V_{\varepsilon,x,j}\coloneqq\{z\in X:\rho_j(z-x)<\varepsilon\}$. 
	Then $V_{\varepsilon,x,j}$ is a neighborhood of $x$ in the $\tau$-topology. Since $I_Q$ is dense in $X$ with respect to the $\tau$-topology, 
	there is $x_\varepsilon\in I_Q$ with $0\leq x_\varepsilon\leq x$ such that $\rho_j(x_\varepsilon-x)<\varepsilon$. 
	Now, $x_\varepsilon\in I_Q$ implies that there are $H_\varepsilon\in\mathscr F(\Gamma)$ and $n_\varepsilon\in\mathbb{N}$ such that
	\begin{equation}\label{3`}
	x_\varepsilon\leq n_\varepsilon\sum_{\gamma\in H_\varepsilon} e_\gamma.
	\end{equation}
	Let
	\begin{equation}\label{4}
	w\coloneqq x  \wedge \sum_{\gamma\in H_\varepsilon} n_\varepsilon e_\gamma.
	\end{equation}
	It follows from $0\leq w\leq\sum\limits_{\gamma\in H_\varepsilon} n_\varepsilon e_\gamma$ and the Riesz decomposition property, that, for each $\gamma\in H_\varepsilon$,
	there exists $y_\gamma$ with
	\begin{equation}\label{5}
	0\leq y_\gamma\leq n_\varepsilon e_\gamma
	\end{equation}
	such that
	\begin{equation}\label{6}
	w=\sum_{\gamma\in H_\varepsilon} y_\gamma.
	\end{equation}
	From (\ref{4}) and (\ref{6}), we have
	\begin{equation}\label{7}
	y_\gamma\leq x ~~~~(\forall\gamma\in H_\varepsilon ).
	\end{equation}
	Also, (\ref{5}) and (\ref{7}) imply that $y_\gamma\leq n_\varepsilon e_\gamma\wedge x$. Now
	\begin{equation}\label{8}
	w=\sum_{\gamma\in H_\varepsilon} y_\gamma\leq\sum_{\gamma\in H_\varepsilon} x \wedge n_\varepsilon e_\gamma = x_{n_\varepsilon, H_\varepsilon}.
	\end{equation} 
	But, from (\ref{3`}) and (\ref{4}), we get
	\begin{equation}\label{9}
	0\leq x_\varepsilon\leq w.
	\end{equation}
	Thus, it follows from (\ref{8}), (\ref{9}), and (\ref{star}), that $0\leq x_\varepsilon\leq x_{n_\varepsilon, H_\varepsilon}\leq x$.
	Hence, $0\leq x-x_{n_\varepsilon, H_\varepsilon}\leq x-x_\varepsilon$ and so $\rho_j(x-x_{n, H})\leq\rho_j(x - x_{n_\varepsilon, H_\varepsilon})\leq\rho_j(x-x_\varepsilon)$ 
	for each $(n,H)\geq ( n_\varepsilon, H_\varepsilon) $. Therefore $x_{n,H}\tc x$.
\end{proof}

The following corollary can be proven easily.
\begin{cor}\label{ut-convergence at q.i.p}
	Let $(X,\tau)$ be a locally solid vector lattice. The following statements are equivalent:
	\begin{enumerate}
		\item\label{qip} $e\in X_+$ is a quasi-interior point;
		\item\label{ut-convergence using q.i.t} for each $x\in X_+,~ x-x\wedge ne\xrightarrow {\mathit{{\tau}}}0$ as $n\rightarrow\infty$.
	\end{enumerate}
\end{cor}

\begin{cor}\label{Lebesgue and weak implies q.i.p}
	Let $(X,\tau)$ be a locally solid vector lattice possessing the $\sigma$-Lebesgue property. Then every weak unit in $X$ is a quasi-interior point. 	
\end{cor}

\begin{proof}
	Let $x\in X^+$, and let $e$ be a weak unit.  Then $x\wedge ne\uparrow x$. So, by the $\sigma$-Lebesgue property,
	we get $x-x\wedge ne\xrightarrow{\mathit{{\tau}}}0$ as $n\rightarrow\infty$.
\end{proof}	

\begin{thm}\label{ut convergence related to TOS}
	Let $(X,\tau)$ be a locally solid vector lattice, and $Q=\{e_\gamma\}_{\gamma\in\Gamma}$ be a topological orthogonal system of $(X,\tau)$. 
	Then $x_\alpha\utc 0$ iff $\abs{x_\alpha}\wedge e_\gamma\tc 0$ for every ${\gamma\in\Gamma}$.
\end{thm}

\begin{proof}
	The forward implication is trivial. For the backward implication, assume $\abs{x_\alpha}\wedge e_\gamma\tc 0$ for every ${\gamma\in\Gamma} $.
	Let $u\in X_+$, $j\in J$. Fix $\varepsilon>0$. We have
	\begin{align*}
	\abs{x_\alpha}\wedge u&=\abs{x_\alpha}\wedge (u-u_{n,H}+u_{n,H})\\
	&\leq\abs{x_\alpha}\wedge (u-u_{n,H} )+\abs{x_\alpha}\wedge u_{n,H}\\
	&\leq (u-u_{n,H})+\abs{x_\alpha}\wedge\sum\limits_{\gamma\in H}u \wedge ne_\gamma\\
	&\leq (u-u_{n,H})+\abs{x_\alpha}\wedge\sum\limits_{\gamma\in H} ne_\gamma\\
	&\leq (u-u_{n,H})+n\big(\abs{x_\alpha}\wedge\sum\limits_{\gamma\in H} e_\gamma\big)\\
	&=(u-u_{n,H})+n\sum\limits_{\gamma\in H}\abs{x_\alpha} \wedge e_\gamma.
	\end{align*}

	Now, Theorem \ref{Proposition II.1.9 in Sch} assures that $u_{n,H}\tc u$, and so, there exists $(n_\varepsilon, H_\varepsilon)\in\mathbb{N}\times\mathscr F(\Gamma)$ such that
	\begin{equation}\label{11}
	\rho_j(u-u_{n_\varepsilon, H_\varepsilon})<\varepsilon.
	\end{equation}
	Thus, $\abs{x_\alpha}\wedge u\leq u-u_{n_\varepsilon, H_\varepsilon}+\sum\limits_{\gamma\in H_\varepsilon}n_\varepsilon(e_\gamma\wedge\abs{x_\alpha})$. 
	But, by the assumption, $e_\gamma\wedge\abs{x_\alpha}\tc 0$ for all $\gamma\in\Gamma$, and so $n_\varepsilon(e_\gamma\wedge\abs{x_\alpha})\tc 0$. 
	Hence, there is $\alpha_{\varepsilon,H_\varepsilon}$ such that
	\begin{equation}\label{12}
	\rho_j\big(n_\varepsilon(e_\gamma\wedge\abs{x_\alpha})\big)<\frac{\varepsilon}{\abs{H_\varepsilon}} \ \ \ \ (\forall\alpha\geq\alpha_{\varepsilon,H_\varepsilon}, \ \forall\gamma\in H_\varepsilon).
	\end{equation}
	Here $\abs{H_\varepsilon}$ denotes the cardinality of $H_\varepsilon$. 
	For $\alpha\geq\alpha_{\varepsilon,H_\varepsilon}$, we have
	\begin{align*}
	\rho_j(\abs{x_\alpha}\wedge u)&\leq\rho_j(u-u_{n_\varepsilon, H_\varepsilon})+\rho_j\big(n_\varepsilon\sum\limits_{\gamma\in H_\varepsilon} \abs{x_\alpha} \wedge e_\gamma\big)\\
	&\leq\varepsilon+\sum\limits_{\gamma\in H_\varepsilon}\rho_j\big(n_\varepsilon(e_\gamma\wedge\abs{x_\alpha})\big)
	<\varepsilon+\sum\limits_{\gamma\in H_\varepsilon}\frac{\varepsilon}{\abs{H_\varepsilon}}=2\varepsilon,
	\end{align*}
	where the second inequality follows from \eqref{11} and the third one from \eqref{12}. 
	Therefore, $\rho_j(\abs{x_\alpha}\wedge u)\to 0$, and so $x_\alpha\utc 0$.
\end{proof}

The following corollary is immediate.

\begin{cor}\label{ut convergence by qip}
	Let $(X,\tau)$ be a locally solid  vector lattice, and $e\in X_+$ be a quasi-interior point. Then $x_\alpha\utc 0$ iff $\abs{x_\alpha}\wedge e\tc 0$.
\end{cor}

Recall that a topological vector space is metrizable iff it has a countable neighborhood base at zero, \cite[Thm. 2.1]{Aliprantis:03}.
In particular, a locally solid vector lattice $(X,\tau)$ is metrizable iff its topology $\tau$ is generated by a countable family $(\rho_k)_{k\in\mathbb{N}}$ 
of Riesz pseudonorms. The following result gives a sufficient condition for the metrizabililty of $u\tau$-topology.
	
\begin{prop}\label{metrizable if has countable t.o.s}
Let $(X,\tau)$ be a complete metrizable locally solid vector lattice. If $X$ has a countable topological orthogonal system, then  the $u\tau$-topology is metrizable.
\end{prop}

\begin{proof}
		First note that, since $(X,\tau)$ is metrizable, $\tau$ is generated by a countable family $(\rho_k)_{k\in\mathbb{N}}$ of Riesz pseudonorms.
		
		Now suppose $(e_n)_{n\in\mathbb{N}}$ to be a topological orthogonal system. 
		For each $n\in\mathbb{N}$, put $d_n(x,y)\coloneqq\sum\limits_{k=1}^\infty\frac{1}{2^k}\frac{\rho_k(\abs{x-y}\wedge e_n)}{1+\rho_k(\abs{x-y}\wedge e_n)}$. 
		Note that each $d_n$ is a semimetric, and $d_n(x,y)\leq 1$ for all $x,y\in X$. If $d_n(x,y)=0$, then $\rho_k(\abs{x-y}\wedge e_n)=0$ for all $k\in\mathbb{N}$, so $(\abs{x-y}\wedge e_n)=0$. 
		For $x,y\in X$, let $d(x,y)\coloneqq\sum\limits_{n=1}^\infty\frac{1}{2^n}d_n(x,y)$. 
		Clearly, $d(x,y)$ is nonnegative and satisfies the triangle inequality, and $d(x,y)=d(y,x)$ for all $x,y\in X$. 
		Now $d(x,y)=0$ iff $d_n(x,y)=0$ for all $n\in\mathbb{N}$ iff $\rho_k(\abs{x-y}\wedge e_n)=0$ for all $k\in\mathbb{N}$ iff $(\abs{x-y}\wedge e_n)=0$ for all $n\in\mathbb{N}$ iff $\abs{x-y}=0$ iff $x=y$. 
		Thus $(X,d)$ is a metric space. 
		Finally, it is easy to see from Theorem \ref{ut convergence related to TOS} that $d$ generates the $u\tau$-topology.
\end{proof}
	
Recall that a topological space $X$ is called {\em submetrizable} if its topology is finer that some metric topology on $X$.	

\begin{prop}\label{weak unit implies submetrizable}
Let $(X,\tau)$ be a metrizable locally solid vector lattice. If $X$ has a weak unit, then the $u\tau$-topology is submetrizable.
\end{prop}

\begin{proof}
Note that, since $(X,\tau)$ is metrizable, then $\tau$ is generated by a countable family $(\rho_k)_{k\in\mathbb{N}}$ of Riesz pseudonorms.
		 
Suppose that $e\in X_+$ is a weak unit. Put $d(x,y)\coloneqq\sum\limits_{k=1}^\infty\frac{1}{2^k}\frac{\rho_k(\abs{x-y}\wedge e)}{1+\rho_k(\abs{x-y}\wedge e)}$. 
Note that $d(x,y)=0$ iff $\rho_k(\abs{x-y}\wedge e)=0$ for all $k\in\mathbb{N}$ iff $\abs{x-y}\wedge e=0$ and, since $e$ is a weak unit, $x=y$. 
It can easily be shown that $d$ satisfies the triangle inequality. 
Assume $x_\alpha\utc x$. Then, for all $u\in X_+$, $\rho_k(\abs{x-y}\wedge u)\to 0$ for all $k\in\mathbb{N}$. In particular, $\rho_k(\abs{x-y}\wedge e)\to 0$ for all $k\in\mathbb{N}$. Then in a similar argument to \cite[p.200]{Vulikh67}, it can be shown that $x_\alpha\dc x$. 
Therefore, the $u\tau$-topology is finer than the metric topology generated by $d$, and hence $u\tau$-topology is submetrizable.
\end{proof}

We do not know whether the converse of propositions \ref{metrizable if has countable t.o.s}, and \ref{weak unit implies submetrizable} is true or not.

\section{Unbounded $\tau$-Completeness}

A subset $A$ of a locally solid vector lattice $(X,\tau)$ is said to be {\em $($sequentially$)$ $u \tau$-complete} if, it is (sequentially) complete in the $u \tau$-topology.
In this section, we relate sequential $u \tau$-completeness of subsets of $X$ with the Lebesgue and Levi properties. First, we remind the following theorem.
	
\begin{thm}\cite[Thm. 1]{Wnuk} \label{Wunk1}
If $( X,\tau)$ is a locally solid vector lattice, then the following statements are equivalent$:$ 
		\begin{enumerate}
			\item $(X,\tau)$ has the Lebesgue and Levi properties$;$
			\item $X$ is $\tau$-complete, and $c_0$ is not lattice embeddable in $(X,\tau)$.
		\end{enumerate}
\end{thm}
	
Recall that two locally solid vector lattices $(X_1,\tau_1)$ and $(X_2,\tau_2)$ are said to be {\em isomorphic}, if there 
exists a lattice isomorphism from $X_1$ onto $X_2$ that is also a homeomorphism; in other	words, if there exists a mapping from $X_1$ onto $X_2$ that 
preserves the algebraic, the lattice, and the topological structures. A locally solid vector lattice $(X_1,\tau_1)$ is said to be {\em lattice embeddable} into another locally solid vector lattice $(X_2,\tau_2)$ if there exists a sublattice $Y_2$ of $X_2$ such that $(X_1,\tau_1)$ and $(Y_2,\tau_2)$ are isomorphic.	
	
Note that $(X,\tau)$ can have the Lebesgue and Levi properties and simultaneously contains $c_0$ as a sublattice, 
but not as a lattice embeddable copy. The following example illustrates this.

\begin{exam}
		Let $s$ denote the  vector lattice of all sequences in $\mathbb{R}$ with coordinatewise ordering. Clearly, $c_0$ is a sublattice of $s$. Define the following separating family of Riesz pseudonorms 
$$
		\mathcal{R}\coloneqq\{\rho_j:\rho_j((x_n)_{n\in\mathbb{N}})\coloneqq\abs{x_j}\}
$$
for each $j\in\mathbb{N}$ and $(x_n)_n\in s$.
Then $\mathcal{R}$ generates a locally solid topology $\tau$ on $s$. It can be easily shown that $(s,\tau)$ has the Lebesgue and Levi properties. Although $c_0$ is a sublattice of $s$, but $(c_0,\norm{\cdot}_\infty)$ is not lattice embeddable in $(s,\tau)$. To see this, consider the sequence $e_n$ 
of the standard unit vectors in $c_0$. Then the sequence $e_n$ is not norm null in $(c_0,\norm{\cdot}_\infty)$, whereas $e_n\tc 0$ in $(s,\tau)$.
\end{exam}
	
\begin{prop}
Let $(X,\tau)$ be a complete locally solid vector lattice. If every $\tau$-bounded subset of $X$ is sequentially $u\tau$-complete, then $X$ has the Lebesgue and Levi properties.
\end{prop}

\begin{proof}
Suppose $X$ does not possess the Lebesgue or Levi properties. 
Then, by Theorem \ref{Wunk1}, $c_0$ is lattice embeddable in $(X,\tau)$.	
Let $s_n=\sum_{k=1}^n e_k$, where $e_k$'s denote the standard unit vectors in $c_0$. 
Clearly, the sequence $s_n$ is norm-bounded in $c_0$ and so it is $\tau$-bounded in $(X,\tau)$. 
Note that $\norm{e_k}_\infty=1\nrightarrow 0$, and so $e_k$ is not $\tau$-null. 
It follows from \cite[Lm. 6.1]{KMT} that $s_n$ is $un$-Cauchy in $c_0$, but is not $un$-convergent in $c_0$. 
That is $s_n$ is $u\tau$-Cauchy which is not $u\tau$-convergent, a contradiction.
\end{proof}
	
Using the proof of the previous result and \cite[Thm. 1$^\prime$]{Wnuk}, one can easily prove the following result.
	
\begin{prop}
Let $X$ be a Dedekind complete vector lattice equipped with a sequentially complete topology $\tau$. 
If every $\tau$-bounded subset of $X$ is sequentially $u\tau$-complete, then $X$ has the $\sigma$-Lebesgue and $\sigma$-Levi properties.
\end{prop}

Clearly, every finite dimensional locally solid vector lattice $(X,\tau)$ is $u\tau$-complete.
On the contrary of \cite[Prop. 6.2]{KMT}, we provide an example of a $\tau$-complete locally solid vector lattice $(X,\tau)$ possessing 
the Lebesgue property such that it is $u\tau$-complete and $\dim X=\infty$.
	
\begin{exam}
Let $X=s$ and $\mathcal{R}=(\rho_j)_{j\in\mathbb{N}}$ such that $\rho_j((x_n))\coloneqq\abs{x_j}$, where $(x_n)_{n \in \mathbb{N}}\in s$. It is easy to see that $(X,\mathcal{R})$ is $\tau$-complete and has the Lebesgue property. Now, we show that  $(X,\mathcal{R})$ is $u\tau$-complete.
Suppose $x^\alpha$ is $u\tau$-Cauchy net. Then, for each $u\in X_+$, we have $\vert x^\alpha-x^\beta\vert\wedge u\tc 0$. Now, $u=u_n$ and, $x^\alpha=x_n^\alpha$. 
Let $j\in\mathbb{N}$, then $\rho_j(\vert x^\alpha-x^\beta\vert\wedge u )\rightarrow 0$ in $\mathbb{R}$ over $\alpha,\beta$
iff $\vert x_j^\alpha-x_j^\beta\vert\wedge u_j\rightarrow 0$ in $\mathbb{R}$ iff $\vert x_j^\alpha-x_j^\beta\vert\rightarrow 0$ in $\mathbb{R}$ over $\alpha,\beta$.\\
Thus, $(x_j^\alpha)_\alpha$ is Cauchy in $\mathbb{R}$ and so there is $x_j\in\mathbb{R}$ such that $x_j^\alpha\rightarrow x_j$ in $\mathbb{R}$ over $\alpha$.
Let $x=(x_j)_{j\in\mathbb{N}}\in s$, then, clearly, $x^\alpha\utc x$.
\end{exam}

\end{document}